\documentclass{amsart}
\usepackage[utf8]{inputenc}
\usepackage{fullpage}
\allowdisplaybreaks
\usepackage{amsmath,amssymb,amsthm}

\usepackage{verbatim}
\usepackage{enumerate}
\usepackage{hyperref}
\usepackage{dsfont}

\usepackage[margin=1in]{geometry} 

\usepackage[subnum]{cases}
\usepackage[noadjust]{cite}
\numberwithin{equation}{section}

\newtheorem{thm}{Theorem}[section]
\newtheorem{lem}[thm]{Lemma}

\newtheorem{prop}[thm]{Proposition}

\newtheorem{ex}[thm]{Example}

\theoremstyle{definition}
\newtheorem{defin}[thm]{Definition}
\newtheorem{rmk}[thm]{Remark}

\renewcommand{\H}{\mathcal{H}}
\newcommand\floor[1]{\lfloor#1\rfloor}
\newcommand\abs[1]{\left|#1\right|}

\newcommand{\D}{\Omega}
\newcommand{\RN}{\mathbb{R}^n}
\newcommand{\RtN}{\mathbb{R}^{2n}}

\newcommand{\vol}{\text{vol}}

\newcommand{\R}{\mathbb{R}}

\newcommand{\Cn}{\mathbb{C}^n}
\newcommand{\Sn}{\mathbb{S}^{2n-1}}

\newcommand{\dbarb}{\overline{\partial}_b}
\newcommand{\dbarbstar}{\overline{\partial}_b^{\ast}}

\newcommand{\boxb}{\square_b}

\newcommand{\hpq}{\mathcal{H}_{p,q}}

\newcommand\numberthis{\addtocounter{equation}{1}\tag{\theequation}}
\makeatletter

\newcommand{\skipitems}[1]{%
	\addtocounter{\@enumctr}{#1}%
}
\makeatletter

\title[Weyl Law for the Kohn Laplacian]{An Analog of the Weyl Law for the Kohn Laplacian on Spheres}

\author{Mohit Bansil}
\address[Mohit Bansil]{Michigan State University, 619 Red Cedar Road C212 Wells Hall, East Lansing, MI 48824, USA}
\email{bansilmo@msu.edu}

\author{Yunus E. Zeytuncu}
\address[Yunus E. Zeytuncu]{University of Michigan--Dearborn, Department of Mathematics and Statistics, 2048 Evergreen Road, Dearborn, MI 48128, USA}
\email{zeytuncu@umich.edu}

\thanks{This work is supported by NSF (DMS-1659203). The work of the second author is also partially supported by a grant from the Simons Foundation (\#353525).}

\date{}
\begin{document}
	
	\begin{abstract}
		We present an explicit formula for the leading coefficient in the asymptotic expansion of the eigenvalue counting function of the Kohn Laplacian on the unit sphere $\Sn$.		
	\end{abstract}
	
	\maketitle
	
	\section{Introduction}
	
	\subsection{Statement} If two planar domains have the same spectrum for the Laplace operator with Dirichlet boundary condition, then are they equivalent up to rigid transformations? Kac's famous paper \cite{Kac} presents this problem of inverse spectral theory to a wide audience. Although the answer is negative \cite{Gordon}, many similar questions wait answers (see \cite{Lu}).
	
	The analogs of these problems on CR manifolds seek to relate the complex geometry of the manifold and the spectrum of the Kohn Laplacian. For example, in \cite{Fu2005} Fu showed that a bounded domain in $\Cn$ is pseudoconvex if and only if the essential spectrum of the Kohn Laplacian on the boundary of the domain is positive.
	
	In this note we look at the eigenvalue counting function for the Kohn Laplacian and obtain the coefficient of the leading term in the asymptotic expansion. Similar estimates on general CR manifolds appear in \cite{Met, Stanton, Ponge}, where the coefficients are expressed in terms of certain integrals. However, here we look at the problem on a specific CR manifold, namely the unit sphere $\Sn\subset \Cn$. We express the leading coefficient explicitly and explore its relation to geometric quantities of $\Sn$.
	
	Let $N(\lambda)$ denote the number of eigenvalues of the Kohn Laplacian on the unit sphere that are less or equal than $\lambda$. In \cite{REU18} it was shown that $N(\lambda)$ grows on the order of $\lambda^{n}$. In this note, we show that 
	\begin{equation}\label{one}
	\lim_{\lambda \to \infty}\frac{N(\lambda)}{\lambda^{n}}=\frac{1}{2^nn!}\left( \sum_{k=1}^\infty k^{-n}\left[\binom{k+n-2}{n-2}+\binom{k-1}{n-2}\right]-\frac{1}{(n-1)^n}\right). 
	\end{equation}
	Although we present an elementary proof of this new result, the main purpose of the paper is to start an investigation for the geometric meaning of these numbers.
	
	\subsection{Weyl's Law}
	The relation between spectral theory and geometry has long been a centerpiece of mathematics. Weyl's law beautifully expresses this relation on bounded domains in $\RN$, see \cite[Section 1.1]{Ivrii}. In particular, if $\D$ is a bounded domain in $\RtN$ and $N(\lambda)$ is the number of non-zero Dirichlet eigenvalues of the Laplace equation on $\D$ that are less than or equal to $\lambda$, counting multiplicity, then Weyl's law states that
	\begin{equation*}
	\lim_{\lambda \to \infty}\frac{N(\lambda)}{\lambda^{n}}=(2\pi)^{2n}\omega_{2n} \vol(\D)    
	\end{equation*}
	where $\omega_d$ is the volume of the unit ball in $\mathbb{R}^d$. More specifically, if $\D$ is the unit ball in $\RtN$ then
	\begin{equation*}
	\lim_{\lambda \to \infty}\frac{N(\lambda)}{\lambda^{n}}=(2\pi)^{2n}\omega_{2n}^2. \end{equation*}
	We see later in the note that the right hand side of \eqref{one} is not as elegant as Weyl's law.
	
	\subsection{Kohn Laplacian}
	Before we go deeper in the calculations we introduce the Kohn Laplacian. The Kohn Laplacian (or $\dbarb$-Laplacian) on a compact orientable CR manifold 
	\begin{equation*}
	\boxb=\dbarb\dbarbstar+\dbarbstar\dbarb
	\end{equation*}
	is a linear, closed, densely defined self-adjoint operator from the space of $(0,q)-$forms $L^2_{(0,q)}(\Sn)$ to itself. Here, $\dbarb$ is a differential operator that send $(0,q)$-forms to $(0,q+1)$-forms, and $\dbarbstar$ denotes its Hilbert space adjoint. Both are densely defined linear operators on respective spaces of differential forms. These operators are fundamental differential operators in the study of CR manifolds. We refer the reader to \cite[Chapter 8]{CS01} for a detailed account.
	
	\subsection{Spectrum on the Sphere}
	
	In \cite{Folland}, Folland computes the eigenvalues and eigenforms of $\boxb$ on $(0,r)-$forms $L^2_{(0,r)}(\Sn)$ by using unitary representations. When $r=0$, in other words on the space of functions, the eigenvalues are of the form 
	\begin{equation}2q(p+n-1)\end{equation}
	with the corresponding eigenspaces $\hpq \left(\Sn\right)$\footnote{For more information on spherical harmonics we refer to \cite{Axler13Harmonic}. We simplify the notation by writing just $\hpq$.} are the space of spherical harmonics, where $p$ and $q$ denote the bi-degrees of the polynomials. It follows from the symmetry of the sphere that the spherical harmonics of different bi-degree are orthogonal. Furthermore, 
	the dimensions of these spaces are given by
	\begin{equation} \label{hpq dim} \dim\left(\hpq (\Sn)\right) = \binom{n + p - 1}{p} \binom{n + q - 1}{q}  - \binom{n + p - 2}{p - 1} \binom{n + q - 2}{q - 1}  \end{equation}
	as computed in \cite{Klima}. We note that when $r=0$, that is when the Kohn Laplacian acts on the space of functions $L^2(\Sn)$, it simplifies to 
	\begin{equation*}
	\boxb=\dbarbstar\dbarb.
	\end{equation*}
	In addition to Folland's original proof, a direct computational proof (for $r=0$) is given in \cite{REU18}.

	\section{Leading Coefficient: Proof of \eqref{one}}
	
	In this section we analyze the asymptotics of the eigenvalue counting function. In particular we show that it has polynomial growth and we compute the leading coefficient in its polynomial expansion. First we define ``big-O notation" in the setting that we will be using it.
	
	\begin{defin}
		
		Given functions $f,g,h$ in the variables on $\mathbb{R}$ we write $f = g + O(h)$ if there are constants $c, x_0\in \R$, so that $\abs{f(x)-g(x)} \leq ch(x)$ for all $x\geq x_0$.

	\end{defin}
	
	Next we have a simple binomial coefficient bound.
	
	\begin{lem}\label{binomial sum bound}
	Let $a,b$ nonnegative integers and $y$ a positive real number. Then
	\begin{align*}
	\sum_{q=1}^{\floor{y}} \binom{q+b}{a} = \frac{y^{a+1}}{(a+1)!} + O(y^{a}). 
	\end{align*}
\end{lem}

\begin{proof}
	Since $a,b$ are fixed $\binom{q+b}{a}$ is a polynomial of degree $a$ in the variable $q$. We have
	\begin{align*}
	\sum_{q=1}^{\floor{y}} \binom{q+b}{a}
	&= \sum_{q=1}^{\floor{y}} \frac{q^{a} }{a!} + O(q^{a-1}) 
	= \frac{(\floor{y})^{a+1}}{(a+1) a!} + O(\floor{y}^{a}) 
	= \frac{(y + O(1))^{a+1}}{(a+1)!} + O({y}^{a})\\ 
	&= \frac{y^{a+1} + O({y}^{a})}{(a+1)!} + O({y}^{a}) 
	= \frac{y^{a+1}}{(a+1)!} + O({y}^{a}),
	\end{align*}
	as desired.		
\end{proof}
	
	We now proceed to the proof of the asymptotics of the eigenvalue counting function. 
	
	\begin{thm} Let $N(\lambda)$ be the counting function as defined above, then
		\begin{equation*}
		\lim_{\lambda \to \infty}\frac{N(\lambda)}{\lambda^{n}}=\frac{1}{2^nn!}\left( \sum_{k=1}^\infty k^{-n}\left[\binom{k+n-2}{n-2}+\binom{k-1}{n-2}\right]-\frac{1}{(n-1)^n}\right). 
		\end{equation*}
	\end{thm}
	\begin{rmk}{}
		In addition to the constant in front of the highest order term $\lambda^n$, the proof also indicates the size of the remainder term; indeed, we show that $N(\lambda)\sim c\lambda^n+O(\lambda^{n-1}\ln \lambda)$. We note that in the asymptotic expansion of the counting function for the Dirichlet Laplacian, the coefficient of the second order term gives the surface area. It would be an interesting further investigation to understand the constant in front of the $\lambda^{n-1}\ln \lambda$ term.
	\end{rmk}
	
	\begin{proof} We note that all the eigenvalues are even integers so
		we define a function $M(x):(0,\infty)\to \mathbb{R}$ by $M(x) = N(2x)$ when $x$ is a positive integer, and $M(x)=M(\floor{x})$ otherwise. We define
		\begin{align*}
		\Delta M(m) = M(m) - M(m-1) = \sum_{2q(p+n-1) = 2m} \dim \H_{p,q}
		= \sum_{q(p+n-1) = m} \dim \H_{p,q}
		= \sum_{pq = m} \dim \H_{p-n+1,q}. 
		\end{align*}
		We note that this quantity is exactly the dimension of the eigenspace of the eigenvalue $2m$.
		Let
		\begin{align*}
		f(p,q) = \frac{(n-1)(p+q)}{(p-n+1)q} \binom{p-1}{p-n}\binom{n+q-2}{q-1},
		\end{align*}
		so that by \ref{hpq dim}, $f(p,q) = \dim \H_{p-n+1,q}$ when $p-n+1,q \geq 1$. Then
		\begin{align*}
		\Delta M(m) = \sum_{pq = m, p \geq n} f(p,q). 
		\end{align*}
		Now
		\begin{align}
		M(x) 
		= \sum_{m\leq x} \Delta M(m) 
		= \sum_{m\leq x}\sum_{pq = m, p \geq n} f(p,q)
		= \sum_{p=n}^x \sum_{q=1}^{\floor{x/p}} f(p,q). \label{eqn: M expr}
		\end{align}
		Now we analyze $f(p,q)$
		\begin{align*}
		f(p,q) = \frac{(n-1)(p+q)}{(p-n+1)q} \binom{p-1}{p-n}\binom{n+q-2}{q-1}
		= \frac{(n-1)}{(p-n+1)} \binom{p-1}{n-1} (\frac pq + 1 )\binom{q+n-2}{n-1}.
		\end{align*}
		We note that
		\begin{align*}
		\frac{(n-1)}{(p-n+1)} \binom{p-1}{n-1} = \binom{p-1}{n-2},
		\end{align*}
		so
		\begin{align*}
		f(p,q) = \binom{p-1}{n-2} (\frac pq + 1 )\binom{q+n-2}{n-1}. 
		\end{align*}
		We split $f(p,q)$. Define
		\begin{align*}
		f_1(p,q) &= \binom{p-1}{n-2} \binom{q+n-2}{n-1} \\
		f_2(p,q) &= \binom{p-1}{n-2} \frac pq\binom{q+n-2}{n-1} 
		\end{align*}
		so that $f = f_1 + f_2$. We analyze both separately. We see
		\begin{align*}
		\sum_{m\leq x} \sum_{pq = m} f_1(p,q)
		&= \sum_{p=1}^x \binom{p-1}{n-2} \sum_{q=1}^{\floor{x/p}} \binom{q+n-2}{n-1}  
		= \sum_{p=1}^x \binom{p-1}{n-2} \left (\frac{(x/p)^{n}}{n!} + O(({x/p})^{n-1})\right ) \\ 
		&= \sum_{p=1}^x \binom{p-1}{n-2}\frac{(x/p)^{n}}{n!} + O( \binom{p-1}{n-2} ({x/p})^{n-1})  
		= \sum_{p=1}^x \binom{p-1}{n-2} \frac{x^{n}}{p^nn!} + O(p^{n-2}(x/p)^{n-1}) \\
		&= \sum_{p=1}^x \binom{p-1}{n-2} \frac{x^{n}}{p^nn!} + O(\frac{x^{n-1}}{p}) 
		= \left (\frac{1}{n!}\sum_{p=1}^x p^{-n} \binom{p-1}{n-2} \right ) x^n+ O({x^{n-1}}\ln x) \numberthis \label{eqn: f1 estimate}
		\end{align*}
		where we have used Lemma \ref{binomial sum bound} for the second equality. 
		Note that $\sum_{p=1}^\infty p^{-n} \binom{p-1}{n-2}$ is a convergent sum. 
		\vskip 0.5 cm
		
		Now we look at
		\begin{align*}
		f_2(p,q)
		= p \binom{p-1}{n-2} \frac 1 q \binom{q+n-2}{q-1} 
		= \frac{p(p-1)!}{(n-2)!(p-n+1)!}\frac{(q+n-2)!}{q(q-1)!(n-1)!}
		= \binom{p}{n-1} \binom{q+n-2}{n-2} .
		\end{align*}
		The idea is to repeat the analysis done for $f_1$ but with the roles of $p,q$ reversed:
		\begin{align*}
		\sum_{m\leq x} \sum_{pq = m} f_2(p,q) 
		= \sum_{q=1}^x \sum_{p=1}^{\floor{x/q}} f_2(p,q)
		= \sum_{q=1}^x \binom{q+n-2}{n-2} \sum_{p=1}^{\floor{x/q}}  \binom{p}{n-1}.  
		\end{align*}
		Hence
		\begin{align*}
		\sum_{m\leq x} \sum_{pq = m} f_2(p,q)
		&= \sum_{q=1}^x \binom{q+n-2}{n-2} \sum_{p=1}^{\floor{x/q}}  \binom{p}{n-1}  
		= \sum_{q=1}^x \binom{q+n-2}{n-2} \left (\frac{({x/q})^{n}}{n!} + O((x/q)^{n-1})\right ) \\
		&= \sum_{q=1}^x \binom{q+n-2}{n-2} \frac{({x/q})^{n}}{n!} + O(\binom{q+n-2}{n-2}(x/q)^{n-1}) \\
		&= \sum_{q=1}^x \binom{q+n-2}{n-2} \frac{x^{n}}{q^nn!} + O(q^{n-2}(x/q)^{n-1}) 
		= \sum_{q=1}^x \binom{q+n-2}{n-2} \frac{x^{n}}{q^nn!} + O(\frac{x^{n-1}}{q}) \\
		&= \left (\frac{1}{n!}\sum_{q=1}^x q^{-n} \binom{q+n-2}{n-2} \right ) x^n+ O({x^{n-1}}\ln x) \numberthis \label{eqn: f2 estimate}
		\end{align*}
		where once again we have used Lemma \ref{binomial sum bound} for the second equality. 
		
		Recall that $f = f_1 + f_2$. By combining \eqref{eqn: f1 estimate} and \eqref{eqn: f2 estimate} we get
		\begin{align*}
		\sum_{m\leq x} \sum_{pq = m} f(p,q) 
		&= \frac{1}{n!}\left (\sum_{q=1}^x q^{-n} \binom{q+n-2}{n-2} + \sum_{p=n}^x  p^{-n}\binom{p-1}{n-2}\right ) x^n+ O({x^{n-1}}\ln x) \\
		&= \frac{1}{n!}\left (\sum_{q=1}^x q^{-n} \binom{q+n-2}{n-2} + \sum_{p=1}^x  p^{-n}\binom{p-1}{n-2}\right ) x^n+ O({x^{n-1}}\ln x) \\
		&= \frac{1}{n!}\left ( \sum_{k=1}^x  k^{-n} \left ( \binom{k+n-2}{n-2} + \binom{k-1}{n-2} \right ) \right ) x^n+ O({x^{n-1}}\ln x). 
		\end{align*}
		Let $h(k) = \binom{k+n-2}{n-2} + \binom{k-1}{n-2}$. We note that $h$ is a polynomial of degree $n-2$. Hence the sum
		\begin{align*}
		\sum_{k=1}^\infty  k^{-n} \left (\binom{k+n-2}{n-2} + \binom{k-1}{n-2} \right )
		\end{align*}
		converges at the same rate as $\sum_{k\geq 1} k^{-2}$. Hence
		\begin{align*}
		\sum_{k=1}^x  k^{-n} \left ( \binom{k+n-2}{n-2} + \binom{k-1}{n-2} \right )
		=\sum_{k=1}^\infty  k^{-n} \left (\binom{k+n-2}{n-2} + \binom{k-1}{n-2} \right ) + O(x^{-1}),
		\end{align*}
		and so we get
		\begin{align}
		\sum_{m\leq x} \sum_{pq = m} f(p,q) 
		= \frac{1}{n!}\left ( \sum_{k=1}^\infty  k^{-n} \left ( \binom{k+n-2}{n-2} + \binom{k-1}{n-2} \right ) \right ) x^n+ O({x^{n-1}}\ln x). \label{eqn: f sum}
		\end{align}
		Recall that in our expression for $M$, \eqref{eqn: M expr}, we had $p \geq n$ in the sum. Hence in order to get an expression for $\Delta M$ we must subtract off
		\begin{align*}
		\sum_{m\leq x} \sum_{pq = m, p < n} f(p,q) = \sum_{p=1}^{n-1} \sum_{q=1}^{\floor{x/p}} f(p,q)
		\end{align*}
		from \eqref{eqn: f sum}. Recall that
		\begin{align*}
		f(p,q) = \binom{p-1}{n-2} (\frac pq + 1 )\binom{q+n-2}{n-1} 
		\end{align*}
		so $f(p,q) = 0$ if $p < n-1$. Hence
		\begin{align*}
		\sum_{p=1}^{n-1} \sum_{q=1}^{\floor{x/p}} f(p,q) 
		&= \sum_{q=1}^{\floor{x/(n-1)}} f(n-1,q) 
		= \sum_{q=1}^{\floor{x/(n-1)}} f_1(n-1,q) + f_2(n-1,q) \\
		&= \sum_{q=1}^{\floor{x/(n-1)}} \binom{n-1-1}{n-2} \binom{q+n-2}{n-1} + \binom{n-1}{n-1} \binom{q+n-2}{n-2} \\
		&= \sum_{q=1}^{\floor{x/(n-1)}} \binom{q+n-2}{n-1} + \sum_{q=1}^{\floor{x/(n-1)}} \binom{q+n-2}{n-2} \\
		&= \frac{(x/(n-1))^n}{n!} + O((x/(n-1))^{n-1}) + \frac{(x/(n-1))^{n-1}}{(n-1)!} + O((x/(n-1))^{n-2}) \\
		&= \frac{x^n}{(n-1)^n n!} + O(x^{n-1}).
		\end{align*}
		Now subtracting our expression for $\sum_{m\leq x} \sum_{pq = m, p < n} f(p,q)$ from \eqref{eqn: f sum} we obtain
		\begin{align*}
		M(x) = \sum_{m\leq x} \sum_{pq = m, p \geq n} f(p,q) 
		= \frac{1}{n!}\left ( \sum_{k=1}^\infty  k^{-n} \left ( \binom{k+n-2}{n-2} + \binom{k-1}{n-2} \right ) \right ) x^n - \frac{x^n}{(n-1)^n n!} + O({x^{n-1}}\ln x),
		\end{align*}
		where we recall our expression for $M$, \eqref{eqn: M expr}. Finally we have
		\begin{equation*}
		\lim_{\lambda \to \infty}\frac{N(\lambda)}{\lambda^{n}}
		= \frac{1}{2^n}\lim_{\lambda \to \infty}\frac{M(\lambda)}{\lambda^{n}}
		=\frac{1}{2^nn!}\left(\sum_{k=1}^\infty k^{-n}\left[\binom{k+n-2}{n-2}+\binom{k-1}{n-2}\right]-\frac{1}{(n-1)^n}\right)
		\end{equation*}
		as desired. 
	\end{proof}
	
	\section{More Explicit Computations}
	
	The problem now becomes to compute the coefficient
	\begin{align*}
	\frac{1}{2^nn!} \left ( \sum_{k=1}^\infty k^{-n} \left (  \binom{k+n-2}{n-2} + \binom{k-1}{n-2} \right ) - \frac{1}{(n-1)^n} \right ) .
	\end{align*} 
	We set $h(k)= \binom{k+n-2}{n-2} + \binom{k-1}{n-2} $ and rewrite the expression above as
	\begin{align*}
	\frac{1}{2^nn!} \left ( \sum_{k=1}^\infty k^{-n} h(k) - \frac{1}{(n-1)^n} \right ) . 
	\end{align*} 
	However we note that $\binom{k+n-2}{n-2} = (-1)^{n} \binom{-k-1}{n-2}$, so $h(k)$ is an even polynomial when $n$ is even and odd when $n$ is odd. Hence $k^{-n}h(k)$ is a polynomial in $k^{-2}$. Since the values of the Riemann zeta function at positive even integers is well-known, then for any fixed $n$ we can easily evaluate the above sum. To see explicitly how the expression above works we will examine the case when $n = 5$. 
	\begin{ex}
		\begin{align*}
		\lim_{\lambda \to \infty}\frac{N(\lambda)}{\lambda^5}&=\frac{1}{2^55!}\left( \sum_{k=1}^\infty k^{-5}\left[\binom{k+3}{3}+\binom{k-1}{3}\right] - \frac{1}{(5-1)^5}\right)\\
		&= \frac{1}{2^55!}\left( \sum_{k=1}^\infty k^{-5}\left[\frac{(k+3)(k+2)(k+1)}{3!}+\frac{(k-1)(k-2)(k-3)}{3!}\right] - \frac{1}{4^5}\right)\\
		&= \frac{1}{2^55!}\left( \sum_{k=1}^\infty \frac{k^{-5}}{6}\left[{k^3 + 6k^2 + 11k + 6}+{k^3 - 6k^2 + 11k - 6}\right] - \frac{1}{4^5}\right)
		= \frac{1}{2^55!}\left( \sum_{k=1}^\infty \frac{2k^{-5}}{6}({k^3 +  11k} ) - \frac{1}{4^5}\right)\\
		&= \frac{1}{2^55!}\left( \sum_{k=1}^\infty \frac{1}{3}({k^{-2} +  11k^{-4}} ) - \frac{1}{4^5}\right)
		= \frac{1}{2^55!}\left( \frac{1}{3}\zeta(2) + \frac{11}{3}\zeta(4)  - \frac{1}{4^5}\right)
		=\frac{1}{2^55!}\left(\frac{\pi^2}{18}+\frac{11\pi^4}{270} - \frac{1}{4^5}\right).
		\end{align*}
	\end{ex}
	
	In general since the values of the Riemann zeta function at positive even integers are rational multiples of a power of $\pi^2$, we conclude with the following proposition.
	
	\begin{prop}
		The leading coefficient of the eigenvalue counting function,
		\begin{align*}
		\lim_{\lambda \to \infty}\frac{N(\lambda)}{\lambda^{n}}=\frac{1}{2^nn!}\left(\sum_{k=1}^\infty k^{-n}\left[\binom{k+n-2}{n-2}+\binom{k-1}{n-2}\right]-\frac{1}{(n-1)^n}\right),
		\end{align*}
		is a rational polynomial in $\pi^2$. 
	\end{prop}
	
	\begin{proof}	
		As in \cite{Loehr2011}, when we let $s$ represent the signed Stirling number of the first kind, $s'$ represent the unsigned Stirling number of the first kind, and $B_l$ represent the $l$-th Bernoulli number we get the following.
		\begin{align*}
		\sum_{k=1}^\infty k^{-n}\left[\binom{k+n-2}{n-2}+\binom{k-1}{n-2}\right] = \sum_{j=0}^{n-1} \frac{(2\pi)^{n-j+1}\abs{B_{n-j+1}}}{ (n-2)!(n-j+1)!}s(n-1, j) 
		\end{align*}	
		By \cite[Theorem 2.74]{Loehr2011} we have
		\begin{align*}
		\binom{k-1}{n-2} 
		&= \frac{(k-1)(k-2)\dots(k-n+2)}{(n-2)!} \\
		&= \frac{k(k-1)(k-2)\dots(k-n+2)}{k(n-2)!} \\
		&= \frac{1}{k(n-2)!} \sum_{j=0}^{n-1} s(n-1, j) k^j
		\end{align*}
		and by \cite[Theorem 2.73]{Loehr2011}
		\begin{align*}
		\binom{k+n-2}{n-2} 
		&= \frac{(k+n-2)(k+n-3)\dots(k+1)}{(n-2)!} \\
		&= \frac{(k+n-2)(k+n-3)\dots(k+1)k}{k(n-2)!}  \\
		&= \frac{1}{k(n-2)!} \sum_{j=0}^{n-1} s'(n-1, j) k^j.
		\end{align*}
		We use the notation
		\begin{align*}
		\chi_{2 \mid l}
		= 
		\begin{cases}
		1, & \text{ if } 2 \mid l \\
		0, & \text{ else }
		\end{cases}
		\end{align*}
		Now we get
		\begin{align*}
		\sum_{k=1}^\infty k^{-n}\left[\binom{k+n-2}{n-2}+\binom{k-1}{n-2}\right]
		&= \sum_{k=1}^\infty \frac{1}{k^{n+1}(n-2)!} \sum_{j=0}^{n-1} (s(n-1, j) + s'(n-1, j)) k^j \\
		&= \sum_{k=1}^\infty \frac{1}{k^{n+1}(n-2)!} \sum_{j=0}^{n-1} (1 + (-1)^{n-1+j}) s(n-1, j) k^{j} \\
		&= \sum_{k=1}^\infty \frac{2}{k^{n+1}(n-2)!} \sum_{j=0}^{n-1} \chi_{2 \mid (n-j+1)} s(n-1, j) k^{j} \\
		&= \sum_{j=0}^{n-1} \frac{2\chi_{2 \mid (n-j+1)}}{(n-2)!}s(n-1, j) \sum_{k=1}^\infty {k^{j-n-1}} \\
		&= \sum_{j=0}^{n-1} \frac{2}{(n-2)!}s(n-1, j) \chi_{2 \mid (n-j+1)} \zeta(n-j+1) \\
		&= \sum_{j=0}^{n-1} \frac{2}{(n-2)!}s(n-1, j) \cdot \frac{(-1)^{\frac{n-j+1}{2}}(2\pi)^{n-j+1}B_{n-j+1}}{2 (n-j+1)!} \\
		&= \sum_{j=0}^{n-1} \frac{(2\pi)^{n-j+1}\abs{B_{n-j+1}}}{ (n-2)!(n-j+1)!}s(n-1, j)
		\end{align*}
		where we have used that $B_l = 0$ if $l$ is odd. 
	\end{proof}
	\vskip 1 cm	
	
	As mentioned in the introduction, the right hand side is not as elegant as Weyl's law. Therefore, we pose the question whether we can hear the surface area with the Kohn Laplacian. We leave the further analysis of these numbers to a future work.
	
	\section*{Acknowledgements} 
	We would like to thank the anonymous referee for constructive feedback.
	We would like to thank Siqi Fu and Tommie Reerink for careful comments on an earlier version of this paper. This research was partially conducted at the NSF REU Site (DMS-1659203) in Mathematical Analysis and Applications at the University of Michigan-Dearborn. We would like to thank the National Science Foundation, National Security Agency, and University of Michigan-Dearborn for their support.

	\bibliographystyle{alpha}
	\bibliography{paper}
\end{document}